\documentclass[preprint,11pt]{elsarticle}

\usepackage{amsfonts, amsmath, amscd}
\usepackage[psamsfonts]{amssymb}

\usepackage{amssymb}

\usepackage{pb-diagram}

\usepackage[all,cmtip]{xy}

\usepackage[usenames]{color}

\newtheorem{theorem}{Theorem}[section]
\newtheorem{lemma}[theorem]{Lemma}
\newtheorem{corollary}[theorem]{Corollary}

\newtheorem{remark}[theorem]{Remark}
\newtheorem{proposition}[theorem]{Proposition}

\newproof{proof}{Proof}

\numberwithin{equation}{section}

\newcommand{\CC}{C_k}
\newcommand{\NN}{\mathbb{N}}

\newcommand{\KK}{\mathcal{K}}

\newcommand{\IR}{\mathbb{R}}

\newcommand{\Ss}{\mathbb{S}}

\newcommand{\e}{\varepsilon}
\newcommand{\cl}{\mathrm{cl}}

\renewcommand{\phi}{\varphi}

\newcommand{\U}{\mathcal U}

\newcommand{\supp}{\mathrm{supp}}
\newcommand{\conv}{\mathrm{conv}}

\input xy
\xyoption{all}

\begin{document}

\begin{frontmatter}

\title{On reflexivity and the Ascoli property for free locally convex spaces}

\author{S.~Gabriyelyan}
\ead{saak@math.bgu.ac.il}
\address{Department of Mathematics, Ben-Gurion University of the Negev, Beer-Sheva, P.O. 653, Israel}

\begin{abstract}
Let $L(X)$ be  the free locally convex space over a  Tychonoff space $X$.
If $X$ is  Dieudonn\'{e} complete (for example, metrizable), then  $L(X)$ is a reflexive group if and only if $X$ is discrete. Answering a question posed in \cite{GKP} we prove also that $L(X)$ is an Ascoli space if and only if $X$ is a countable discrete space.
\end{abstract}


\end{frontmatter}


\section{Introduction}


Let $G$ be an abelian topological group.  A {\em character} $\chi$ of $G$  is a continuous homomorphism from $G$ into the unit circle group $\Ss:=\{ z\in \mathbb{C}: |z|=1\}$. The {\em dual group} $\widehat{G}$ of $G$ is the group of all characters of $G$. We denote by $G^\wedge$ the group $\widehat{G}$ endowed with the compact-open topology. The map $\alpha_G : G\to G^{\wedge\wedge} $, $g\mapsto (\chi\mapsto \chi(g))$, is called {\em the canonical homomorphism}. If $\alpha_G$ is a topological isomorphism the group $G$ is called {\em  reflexive}.

The celebrated Pontryagin--van Kampen Duality Theorem states that every locally compact abelian group is reflexive. It is well known that the class of reflexive groups is closed under taking arbitrary products (\cite{Kap}).  Banach spaces and  reflexive locally convex spaces considered as abelian topological groups are reflexive groups (\cite{Smith}). In particular, the classical Banach space $\ell_1$ is a reflexive group which is not a reflexive space.

For a Tychonoff space $X$, we denote by $C_p(X)$ and $\CC(X)$ the space $C(X)$ of all continuous functions on $X$ endowed with the pointwise topology $\tau_p$ and the compact-open topology $\tau_k$, respectively. Note that $\CC(X)$ is a reflexive locally convex space if and only if $X$ is discrete, see \cite[Theorem~11.7.7]{Jar}. Hern\'{a}ndez and Uspenski\u{\i} proved in \cite{HerUsp} that $\CC(X)$ is a reflexive group for every Ascoli $\mu$-space $X$.
Recall that $X$ is a {\em $\mu$-space } if every functionally bounded subset of $X$ has compact closure, and, following \cite{HerUsp} (see also Remark \ref{rem:Def-Ascoli} below), $X$ is called an {\em  Ascoli space} if every compact subset $\KK$ of $\CC(X)$  is equicontinuous. It is also shown in \cite{HerUsp} that if $C_p(X)$ is a reflexive group, then $X$ is a $P$-space; moreover, the question of whether $C_p(X)$ is a reflexive group is undecidable in $ZFC$ even for simple spaces.

Recall that a Tychonoff space $X$ is {\em Dieudonn\'{e} complete} if the universal uniformity $\U_X$ on $X$ is complete. For numerous characterizations of Dieudonn\'{e} complete spaces see Section 8.5.13 of \cite{Eng}.  The definitions of Fr\'{e}chet--Urysohn, sequential, paracompact and $k$-spaces also can be  found in \cite{Eng}. Relations between the aforementioned topological notions are given in the following diagram

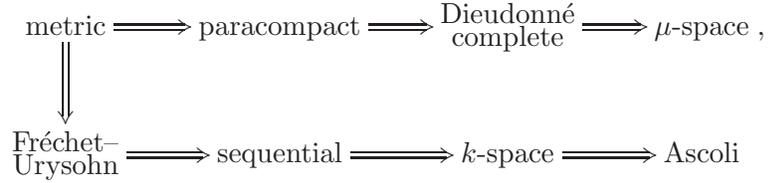
\[
\xymatrix{
\mbox{metric} \ar@{=>}[r]\ar@{=>}[d]   &   \mbox{paracompact} \ar@{=>}[r] &  {\substack{\mbox{Dieudonn\'{e}} \\ \mbox{complete}}}  \ar@{=>}[r] &  \mbox{$\mu$-space} \\
{\substack{\mbox{Fr\'{e}chet--} \\ \mbox{Urysohn}}} \ar@{=>}[r] & \mbox{sequential} \ar@{=>}[r] & \mbox{$k$-space} \ar@{=>}[r] & {\mbox{Ascoli}}
},
\]
and none of these implications is reversible.

One of the most important classes of locally convex spaces is the class of free locally convex spaces. Following \cite{Mar}, the {\em  free locally convex space}  $L(X)$ on a Tychonoff space $X$ is a pair consisting of a locally convex space $L(X)$ and  a continuous map $i: X\to L(X)$  such that every  continuous map $f$ from $X$ to a locally convex space  $E$ gives rise to a unique continuous linear operator ${\bar f}: L(X) \to E$  with $f={\bar f} \circ i$. The free locally convex space $L(X)$ always exists and is essentially unique. For every Tychonoff space $X$, the set $X$ forms a Hamel basis for $L(X)$ and  the map $i$ is a topological embedding, see \cite{Rai,Usp2}, so we shall identify $x\in X$ with its image $i(x)\in L(X)$. Note that a Tychonoff space $X$ is Ascoli if and only if the canonical map $L(X)\to \CC\big(\CC(X)\big)$ is an embedding of locally convex spaces, see \cite{Gabr-LCS-Ascoli}.

It is well known that a reflexive locally convex space $E$ must be barrelled. Since $L(X)$ is a barrelled space if and only if $X$ is discrete (see \cite{GM} or \cite{Gabr-L(X)-Mackey}, for a more general assertion), it easily follows that $L(X)$ is a reflexive locally convex space if and only if $X$ is discrete. As we mentioned above there are non-reflexive locally convex spaces which are reflexive groups. Therefore it is natural to ask: {\em For which Tychonoff spaces $X$, the free locally convex space $L(X)$ is a reflexive group}? Below we obtain a partial answer to this question.

\begin{theorem} \label{t:L(X)-reflexive}
Let $X$ be a Dieudonn\'{e} complete space. Then $L(X)$ is a reflexive group if and only if $X$ is discrete.
\end{theorem}

Let $X$ be a Tychonoff space.
We proved in \cite{Gabr} that $L(X)$ is a $k$-space if and only if $X$ is a countable discrete space (in this case $L(X)$ is a sequential non-Fr\'{e}chet--Urysohn space). This result motivates the following question posed in \cite{GKP}: {\em Is it true that $L(X)$ is an Ascoli space  only if $X$ is a countable discrete space}? Banakh obtained in \cite[Theorem~10.11.9]{Banakh-Survey} a partial answer to this question by showing that if $X$ is a Dieudonn\'{e} complete space such that $L(X)$ is an Ascoli space, then $X$ is countable and discrete.
Let $G$ be an abelian topological group. It is well known (and easy to check) that $\alpha_G$ is continuous if and only if every compact subset of the dual group $G^\wedge$ is equicontinuous.
Therefore, if $\alpha_G$ is {\em not} continuous, then $G$ is not an Ascoli space by Proposition \ref{p:group-Ascoli-continuous} below.
Essentially using this simple remark and the proof of Theorem \ref{t:L(X)-reflexive},  in Proposition \ref{p:Ascoli-L(X)} we give an independent, short and clear proof of Banakh's result. Also the condition of being a Dieudonn\'{e} complete space can be easily removed and, therefore,  we obtain an affirmative answer to the above question.

\begin{theorem} \label{t:Ascoli-L(X)}
For a Tychonoff space $X$, the space $L(X)$ is Ascoli if and only if $X$ is a countable discrete space.
\end{theorem}


\section{Proofs} \label{sec:2}


We start from some necessary definitions and notations. Set $\NN:=\{ 1,2,\dots\}$. 
The closure of a subset $A$ of a Tychonoff space $X$ is denoted by $\overline{A}$ or $\cl(A)$. 
The {\em Dieudonn\'{e} completion} $\mu X$ of $X$ is the completion of the uniform space $(X,\U_X)$.
The support of a function $f\in C(X)$ is denoted by $\supp(f)$. Recall that the sets
\[
[K;\e] :=\{ f\in C(X): |f(x)|<\e \; \forall x\in K\},
\]
where $K$ is a compact subset of $X$ and $\e>0$, form a base at zero of the compact-open topology $\tau_k$ of $\CC(X)$.

Let $E$ be a locally convex space. We denote by $E'$ the dual space of $E$. The polar of a subset $A$ of $E$ is denoted by $A^\circ :=\{ \chi\in E': |\chi(x)|\leq 1 \, \forall x\in A\}$.

\begin{remark} \label{rem:Def-Ascoli} {\em
Following \cite{BG}, a  Tychonoff (Hausdorff) space $X$ is called an {\em Ascoli space} if each compact subset $\KK$ of $\CC(X)$ is evenly continuous. In other words, $X$ is Ascoli if and only if the compact-open topology of $\CC(X)$ is Ascoli in the sense of \cite[p.45]{mcoy}. It is noticed in \cite{Gabr-LCS-Ascoli} that this definition coincides with the definition of Ascoli spaces from \cite{HerUsp} given in Introduction.} \qed
\end{remark}

Noble showed in \cite{Nob2} that $\alpha_G$ is continuous if the group $G$ is a $k$-space. Below we generalize this result.

\begin{proposition} \label{p:group-Ascoli-continuous}
If an abelian topological group $G$ is an Ascoli space, then the canonical homomorphism $\alpha_G$ is continuous.
\end{proposition}

\begin{proof}
We have to show that every compact subset $K$  of $G^\wedge$ is equicontinuous. Since $G^\wedge$ is a subgroup of the group $\CC(G,\Ss)$  of all continuous functions from $G$ to $\Ss$ endowed with the compact-open topology, we obtain that $K$ is a compact subset of  $\CC(G,\Ss)$. Therefore $K$ is evenly continuous by 
Corollary 5.3 of \cite{BG}. To show that $K$ is equicontinuous, fix $\e>0$. For every $\chi\in K$, choose a neighborhood $U_\chi$ of $\chi\in K$ and a neighborhood $O_\chi$ of zero in $G$ such that
$
|\eta(g)-1|<\e$ for all $g\in O_\chi$ and $\eta\in U_\chi.
$
Since $K$ is compact, we can find $\chi_1,\dots,\chi_n\in K$ such that $\bigcup_{i=1}^n U_{\chi_i} = K$. Set $O:=\bigcap_{i=1}^n O_{\chi_i}$. Then
$
|\eta(g)-1|<\e$ for all $ g\in O $ and $\eta\in K.
$
Thus $K$ is equicontinuous. \qed
\end{proof}

Denote by $M_c(X)$ the space of all real regular Borel measures on $X$ with compact support. It is well known that the dual space of $\CC(X)$ is $M_c(X)$, see \cite[Proposition~7.6.4]{Jar}. For every $x\in X$, we denote by $\delta_x \in M_c(X)$ the evaluation map (Dirac measure), i.e. $\delta_x(f):=f(x)$ for every $f\in C(X)$. The total variation norm of a measure $\mu \in M_c(X)$ is denoted by $\| \mu\|$. Denote by $\tau_e$ the polar topology on $M_c(X)$ defined by the family of all equicontinuous  pointwise bounded subsets of $C(X)$.
We shall use the following deep result of Uspenski\u{\i} \cite{Usp2}.

\begin{theorem}[\cite{Usp2}] \label{t:Free-complete-L}
Let $X$ be a Tychonoff space and let $\mu X$ be the Dieudonn\'{e} completion of $X$. Then the completion $\overline{L(X)}$ of $L(X)$ is topologically isomorphic to $\big(M_c(\mu X),\tau_e\big)$.
\end{theorem}
In what follows we shall also identify elements $x\in X$ with the corresponding Dirac measure $\delta_x \in M_c(X)$.
We need the following corollary of Theorem \ref{t:Free-complete-L} noticed in \cite{Gabr-L(X)-Mackey}.
\begin{corollary}[\cite{Gabr-L(X)-Mackey}] \label{p:Ck-Mc-compatible}
Let $X$ be a  Dieudonn\'{e} complete space. Then $(M_c(X),\tau_e)' =\CC(X)$.
\end{corollary}

We shall use the following fact, see Lemmas 5.10.2 and 5.10.3 and Theorem 5.10.4 of \cite{NaB}  (this fact is also proved in the ``if'' part of the Ascoli theorem \cite[Theorem 3.4.20]{Eng}).
\begin{proposition} \label{p:Ascoli-Free-Ck}
Let $X$ be a Tychonoff space and $A$ be an equicontinuous pointwise bounded subset of $C(X)$. Then the $\tau_p$-closure ${\bar A}$ of $A$ is $\tau_k$-compact and equicontinuous. 
\end{proposition}

The following proposition is used to prove Theorem \ref{t:L(X)-reflexive}. 
\begin{proposition}[\cite{Gab-Respected}] \label{p:bounded-in-Mc}
Let $X$ be a  Dieudonn\'{e} complete space and let $\KK$ be a $\tau_e$-closed subset of $M_c(X)$. Then  $\KK$ is $\tau_e$-compact if and only if
 there is a compact subset $C$ of $X$ and $\e>0$ such that $\KK\subseteq [C;\e]^\circ$.
\end{proposition}

We need also the next simple assertion.
\begin{lemma} \label{l:sequence-L(X)}
Let $X$ be a Tychonoff space containing an infinite compact subset $K$. Let $S=\{ \eta_n:n\in\NN\}$ be a  one-to-one sequence in $K$ and let $\eta_0$ be a cluster point of $S$ such that $\eta_0\not\in S$. Then for each sequence $\{ \lambda_k: k\in\NN\}$ of positive numbers  such that $\sum_{k=1}^\infty \lambda_k =1$, the sequence
\[
\chi_n := \sum_{k=1}^n \lambda_k \eta_k \in L(X), \quad n\in\NN,
\]
converges to $\chi_0 :=\sum_{k=1}^\infty \lambda_k \eta_k$ in $(M_c(X),\tau_e)$.
\end{lemma}

\begin{proof}
Let $Q$ be a pointwise bounded and equicontinuous subset of $C(X)$. Then $Q|_K := \{ f|_K: f\in Q\}$ is a pointwise bounded and equicontinuous subset of $C(K)$. Proposition \ref{p:Ascoli-Free-Ck} implies that the pointwise closure of $Q|_K$ is a compact subset of the Banach space $C(K)$. Thus $Q|_K $ is uniformly bounded, i.e., there is a $B>0$ such that
\[
|f(x)|\leq B, \quad \forall x\in K, \; \forall f\in Q.
\]
Choose an $n_0\in \NN$ such that $B\cdot \sum_{k> n_0} \lambda_k <1$. Then
\[
|(\chi_0 -\chi_n)(f)| \leq \sum_{k>n} \lambda_k  |f(\eta_k)|\leq \sum_{k>n} \lambda_k B   <1,  \quad \forall f\in Q,
\]
and hence $\chi_n \in \chi_0 + Q^\circ$ for every $n>n_0$. Thus $\chi_n \to \chi_0$ in $(M_c(X),\tau_e)$ as desired. \qed
\end{proof}

Recall that a locally convex space $E$ is a {\em semireflexive group} if the canonical homomorphism $\alpha_E$ is bijective (but not necessarily continuous).
\begin{proposition} \label{p:L(X)-semireflexive}
If a Tychonoff space $X$ contains an infinite compact subset $K$, then $L(X)$ is not a semireflexive group.
\end{proposition}

\begin{proof}
Suppose for a contradiction that $L(X)$  is a semireflexive group. Then the closed convex hull $C:=\overline{\conv} (K)$ of $K$  is a weakly compact subset of $L(X)$ by Theorem 2.1 of \cite{HerUsp}, and therefore, $C$ is a compact subset of $L(X)$ by 
Theorem 1.2 of \cite{Gab-Respected}. In particular, the sequence $\{ \chi_n:n\in\NN\}$, constructed in Lemma \ref{l:sequence-L(X)}, has compact closure in $L(X)$. But since $\chi_n \to \chi_0$ in $(M_c(\mu X),\tau_e)=\overline{L(X)}$ by Lemma \ref{l:sequence-L(X)} and $\chi_0 \not\in L(X)$, we obtain a contradiction. \qed
\end{proof}

Now we are ready to prove Theorem \ref{t:L(X)-reflexive}.

\medskip
{\em Proof of Theorem \ref{t:L(X)-reflexive}.}
Assume that $L(X)$ is a reflexive group. We have to prove that $X$ is discrete. By Proposition \ref{p:L(X)-semireflexive}, we shall assume that all  compact subsets of $X$ are finite.

Suppose for a contradiction that the space $X$ is not discrete.
Proposition 2.4 of \cite{Gabr-L(X)-Mackey}
states that there exist an infinite cardinal $\kappa$, a point $z\in X$, a family $\{ g_i\}_{i\in\kappa}$ of continuous functions from $X$ to $[0,2]$ and a family $\{ U_i\}_{i\in\kappa}$ of open subsets of $X$ such that
\begin{enumerate}
\item[{\rm $(\alpha)$}] $\supp(g_i) \subseteq U_i $ for every $i\in\kappa$;
\item[{\rm $(\beta)$}] $U_i\cap U_j=\emptyset $ for all distinct $i,j\in\kappa$;
\item[{\rm $(\gamma)$}] $z\not\in U_i$ for every $i\in\kappa$ and $z\in \cl\big(\bigcup_{i\in\kappa} \{ x\in X: g_i(x)\geq 1\}\big)$.
\end{enumerate}
Set $E:=L(X)=(M_c(X),\tau_e)$.
By Corollary \ref{p:Ck-Mc-compatible}, we have $E'=(M_c(X),\tau_e)'=C(X)$. So we can consider the family $\{ g_i\}_{i\in\kappa}$ as a subset of the dual space $E'=C(X)$ of $E$. Let $\tau_k$ be the compact-open topology on $E'$. Denote by $\mathbf{0}$ the zero function on $X$.

\smallskip
{\em Claim 1. The set $K:= \{ g_i\}_{i\in\kappa} \cup\{ \mathbf{0}\}$ is a compact subset of $(E',\tau_k)$.}

\smallskip
Indeed, fix arbitrarily a compact subset $Z$ of $E$ and $\delta>0$.  We have to show that all but finitely many of functions $g_i$s belong to the neighborhood $[Z;\delta]$ of $\mathbf{0}$ in $(E',\tau_k)$.  By Proposition \ref{p:bounded-in-Mc}, we can assume that $Z=[C;\e]^\circ$, where $C$ is a compact subset of $X$ and $\e>0$. Therefore
\[
\begin{aligned}
Z=[C;\e]^\circ & = \big\{ \mu\in M_c(X): |\mu(f)|\leq 1 \mbox{ for every } f\in [C;\e]\big\} \\
& = \big\{ \mu\in M_c(X): \supp(\mu) \subseteq C \mbox{ and } \| \mu\|\leq 1/\e \big\}.
\end{aligned}
\]
Since the compact set $C$ is finite, $(\alpha)$ and $(\beta)$ imply that $g_i|_C =0$ and hence $g_i\in [Z;\delta]$ for all but finitely many indices $i\in\kappa$. Thus $K$ is compact and Claim 1 is proved.

\smallskip
{\em Claim 2. The compact set $K$ is not equicontinuous.}

\smallskip
Indeed, fix arbitrarily a $\tau_e$-neighborhood
\[
W=[\KK;\e]= \big\{ \mu\in M_c(X): |\mu(f)|<\e  \mbox{ for every }  f\in\KK\big\}
\]
of zero in $M_c(X)$, where $\KK$ is a pointwise bounded and equicontinuous subset of $C(X)$ and $\e>0$. Choose a neighborhood $V$ of the point $z$ in $X$ such that
\[
|f(x)-f(z)|<\e/2, \quad  \mbox{ for every }  x\in V  \mbox{ and each }   f\in\KK.
\]
Then
$
|(\delta_x-\delta_z)(f)|<\e/2$ for every $ x\in V$ and each $f\in\KK.
$
Therefore $\delta_x-\delta_z \in W$ for every $x\in V$. By $(\gamma)$, there are $i\in\kappa$ and $x_i\in U_i$ such that $x_i\in V$ and $g_i(x_i)\geq 1$. Hence (note that $g_i(z)=0$ for every $i\in\kappa$  since $z\not\in U_i$)
\[
|(\delta_{x_i}-\delta_z)(g_i)|=g_i(x_i) - g_i(z)=g_i(x_i) \geq 1.
\]
Thus $K$ is not equicontinuous and Claim 2 is proved.

Now Claims 1 and 2 imply that  $(E',\tau_k)$ contains a compact subset $K$ which is not equicontinuous. Thus $\alpha_{L(X)}$ is not continuous by 
Proposition 2.3 of \cite{HerUsp}. Thus $L(X)$ is not a reflexive group. This contradiction shows that $X$ must be discrete as desired.

\smallskip
Conversely, if $X$ is discrete, then $L(X)$ is complete and barrelled by Theorem \ref{t:Free-complete-L} and Theorem 6.4 of \cite{GM}, respectively. Therefore, $L(X)$ is a reflexive group by Corollary 2.5 of \cite{HerUsp}.  \qed


\smallskip
We shall use the following proposition to show that a space is not Ascoli.
\begin{proposition}[\cite{GKP}] \label{p:Ascoli-sufficient}
Assume  that a Tychonoff space $X$ admits a  family $\U =\{ U_i : i\in I\}$ of open subsets of $X$, a subset $A=\{ a_i : i\in I\} \subseteq X$ and a point $z\in X$ such that
\begin{enumerate}
\item[{\rm (i)}] $a_i\in U_i$ for every $i\in I$;
\item[{\rm (ii)}] $\big|\{ i\in I: C\cap U_i\not=\emptyset \}\big| <\infty$  for each compact subset $C$ of $X$;
\item[{\rm (iii)}] $z$ is a cluster point of $A$.
\end{enumerate}
Then $X$ is not an Ascoli space.
\end{proposition}

\smallskip
The next proposition is proved by T.~Banakh in \cite[Theorem~10.11.9]{Banakh-Survey}, below we provide an independent and much simpler and clearer proof of this assertion. 
\begin{proposition}[\cite{Banakh-Survey}] \label{p:Ascoli-L(X)}
Let $X$ be a Dieudonn\'{e} complete space. If $L(X)$ is an Ascoli space, then $X$ is a countable discrete space.
\end{proposition}

\begin{proof}
We prove the proposition in three steps.

\smallskip
{\em Step A. The space $X$ does not contain infinite compact subsets.}

\smallskip
Suppose for a contradiction that $X$ contains an infinite compact subset $K$. Consider two sequences $\{ \eta_n:n\geq 0\}\subseteq K$ and $\{ \chi_n:n\in\NN\} \subseteq L(X)$ constructed in Lemma \ref{l:sequence-L(X)}.
For every $n,m\in\NN$, set
\[
a_{n,m} := \frac{1}{m+1} \chi_n + m(\eta_n -\eta_0) ,
\]
and put $A:=\{ a_{n,m}: n,m\in\NN\}$ and $z:=0$.

\smallskip

{\em Claim A.1. $0\in \overline{A}\setminus A$.}

\smallskip
For every $n\in\NN$, take a continuous function $f_n: X\to [0,1]$ such that
\begin{equation} \label{equ:L(X)-Ascoli-11}
\eta_n (f_n)= f_n(\eta_n)=1 \; \mbox{ and } \; \eta_0 (f_n)= f_n(\eta_0)=0.
\end{equation}
Then, for every $n\in\NN$, we have
\begin{equation} \label{equ:L(X)-Ascoli-12}
|\chi_n(f_n)|\leq \sum_{k=1}^n \lambda_k \cdot \eta_k(f_n) \leq \sum_{k=1}^n \lambda_k \cdot f_n(\eta_k) \leq \sum_{k=1}^n \lambda_k <1.
\end{equation}
Now (\ref{equ:L(X)-Ascoli-11}) and (\ref{equ:L(X)-Ascoli-12}) imply
\[
|a_{n,m} (f_n)|=\left| \frac{1}{m+1} \chi_n(f_n)+ m\cdot \eta_n(f_n)\right|\geq m-\frac{1}{m+1}>0,
\]
and hence $0\not\in A$. To show that $0\in\overline{A}$ take arbitrary a neighborhood $W$ of $0$ in $(M_c(X),\tau_e)$ and choose an absolutely convex  neighborhood $U$ of $0$ in $(M_c(X),\tau_e)$ such that $3U\subseteq W$. By Lemma \ref{l:sequence-L(X)}, choose an $N\in\NN$ such that $\chi_n \in \chi_0 +U$ for every $n\geq N$, and choose an $m_0\in \NN$ such that $\frac{1}{m_0+1} \chi_0 \in U$. Now, since $\eta_0$ is a cluster point of the sequence $\{ \eta_n:n\in\NN\}$, for the chosen $N$ and $m_0$, take an $n_0 >N$ such that $m_0(\eta_{n_0} - \eta_0) \in U$. Noting that $\frac{1}{m_0+1} U \subseteq U$ we obtain
\[
a_{n_0,m_0} = \frac{1}{m_0+1} \big( \chi_{n_0}-\chi_0 \big) + \frac{1}{m_0+1} \chi_0 + m_0(\eta_{n_0}-\eta_0) \in U+U+U \subseteq W.
\]
Thus $0\in \overline{A}$ and the claim is proved. \qed

\smallskip
To show that $L(X)$ is not Ascoli we shall use Proposition \ref{p:Ascoli-sufficient}. To this end, we have to find also a sequence $\U =\{ U_n : n\in \NN\}$ of open subsets of $L(X)$ such that (i) and (ii) of Proposition \ref{p:Ascoli-sufficient} are satisfied. Below we construct such a family $\U$.

\smallskip
Let $p:X\to \beta X$ be a natural embedding of $X$ into the Stone--\v{C}ech compactification $\beta X$ of $X$. Then $p$ extends uniquely to a continuous linear injective operator from $L(X)$ to $L(\beta X)$ which is also denoted by $p$. For every $n\in\NN$, denote by $F_n : \beta X \to [0,1]$ the unique extension of $f_n$ onto $\beta X$. Then, by   (\ref{equ:L(X)-Ascoli-11}) and (\ref{equ:L(X)-Ascoli-12}), we have
\begin{equation} \label{equ:L(X)-Ascoli-13}
|p(\chi_n)(F_n)|  = \left|\sum_{k=1}^n \lambda_k \cdot F_n\big(p(\eta_k)\big) \right| = \left| \sum_{k=1}^n \lambda_k \cdot f_n(\eta_k)\right| <1,
\end{equation}
\begin{equation} \label{equ:L(X)-Ascoli-14}
|p(\eta_n)(F_n)|  = |F_n\big(p(\eta_n)\big)| =|\eta_n(f_n)|=1,
\end{equation}
and
\begin{equation} \label{equ:L(X)-Ascoli-15}
p(\eta_0)(F_n)=F_n\big(p(\eta_0)\big)=\eta_0(f_n)=0 \quad \mbox{ for every } n\in\NN.
\end{equation}

Set $E:= C(\beta X)$ and $H:=L(\beta X)$. Lemma \ref{l:sequence-L(X)} implies that $p(\chi_n)\to p(\chi_0)\in M_c(\beta X)\setminus H$. Hence, for every $m\in\NN$,  the sequence  $S_m=\big\{ \frac{1}{m+1}p(\chi_n) :n\in\NN\big\}$ is closed and discrete in $H$. Note that $H$ is a Lindel\"{o}f space  by Proposition 5.2 of \cite{GM}. Therefore $S_m$ is $C$-embedded in $H$, i.e. every real-valued function on $S_m$ can be extended to a continuous function on the whole $H$. Take a continuous function $G_m:H\to \IR$ such that $G_m\big(\frac{1}{m+1}p(\chi_n)\big) = n$ for every $n\in\NN$. For every $n,m\in\NN$, set
\[
\widetilde{V}_{n,m}^\beta :=G_m^{-1}\big( (n-0.1,n+0.1)\big) -\frac{1}{m+1}p(\chi_n).
\]
Then, for every $m\in\NN$, the family
\[
\mathcal{V}_m^\beta  :=\left\{ \frac{1}{m+1} p(\chi_n) + \widetilde{V}_{n,m}^\beta : \; n\in\NN\right\}
\]
of open subsets of $H$ is discrete in $H$ (i.e., every $h\in H$ has a neighborhood $U_h$  which intersects with at most one element of the family $\mathcal{V}_m^\beta $).

For every $n,m\in\NN$, set
\begin{equation} \label{equ:L(X)-Ascoli-16}
W_{n,m}^\beta  := \left\{ h\in H: |h(x_n)| < 1 \right\} \; \mbox{ and } \; V_{n,m}^\beta  :=\widetilde{V}_{n,m}^\beta  \cap W_{n,m}^\beta ,
\end{equation}
and put
\[
U_{n,m}^\beta := p\big(a_{n,m}\big) +  V_{n,m}^\beta.
\]

\smallskip
{\em Claim A.2. $\big|\{ (n,m)\in\NN\times\NN: K\cap U_{n,m}^\beta \not=\emptyset\}\big|<\infty$ for every compact subset $K$ of $H$.}

\smallskip
Indeed, let $K$ be a compact subset of $H$. Then $K$ is compact and hence bounded in $\big(M_c(\beta X),\tau_e\big)$. Since $\big(M_c(\beta X),\tau_e\big)'=E$ by Corollary \ref{p:Ck-Mc-compatible}, we obtain that $K$ is weak-$\ast$ bounded in $E'$. As $E$ is a Banach space, $K$ is equicontinuous and hence there is an open neighborhood $\mathcal{O}$ of zero in $E$ such that $K\subseteq \mathcal{O}^\circ$.

Set  $Z:=\{ F_n: n\in\NN\} \subseteq E$. It is clear that $Z$ is a bounded subset of $E$, so we can find a $C>0$ such that $Z\subseteq C\cdot\mathcal{O}$. Then
\begin{equation} \label{equ:L(X)-Ascoli-17}
\mathcal{O}^\circ \subseteq C \cdot Z^\circ = \{ \chi\in M_c(\beta X) : |\chi(F_n)|\leq C  \mbox{ for all }   n\in\NN\}.
\end{equation}
Now, if $\chi=p\big(a_{n,m}\big)+h\in U_{n,m}^\beta$ with $h\in V_{n,m}^\beta,$  (\ref{equ:L(X)-Ascoli-13})-(\ref{equ:L(X)-Ascoli-16}) imply
\[
\begin{aligned}
|\chi(F_n)| & \geq \big|p\big(a_{n,m}\big)(F_n)\big| -|h(F_n)| \\
& \geq m|p(\eta_n)(F_n)| - \frac{1}{m+1} |p(\chi_n)(F_n)| - |h(F_n)|> m-2.
\end{aligned}
\]
Therefore, by (\ref{equ:L(X)-Ascoli-17}), if $m>C+2$ then $K\cap U_{n,m}^\beta \subseteq \mathcal{O}^\circ \cap U_{n,m}^\beta =\emptyset$. For a fixed natural number $m\leq C+2$, $K\cap U_{n,m}^\beta $ is nonempty only for a finite number of $n$  because the family $\mathcal{V}_m^\beta$ is discrete. The claim is proved. \qed

\smallskip
Set $\U :=\{ p^{-1}(U^\beta_{n,m}): n,m\in\NN\}$. Then Claims A.1 and A.2 imply that the families $A$, $\U$ and $z=0\in L(X)$ satisfy (i)-(iii) of Proposition \ref{p:Ascoli-sufficient} as well. Thus $L(X)$ is not an Ascoli space. This contradiction finishes the proof of Step A. Thus we assume in the next step that $X$ does not contain infinite compact subsets and hence $L(X)=(M_c(X),\tau_e)$ by Theorem \ref{t:Free-complete-L}.

\medskip
{\em Step B. The space $X$ is discrete.}

\smallskip
Suppose for a contradiction that $X$ is not discrete. Then Claims 1 and 2 in the proof of Theorem \ref{t:L(X)-reflexive} imply that the space $\CC\big( L(X)\big)$ contains a non-equicontinuous compact subset. Thus the space $L(X)$ is not Ascoli. This contradiction shows that $X$ must be discrete.

\smallskip
{\em Step C.} By Step $B$, the space $X$ is discrete. Therefore $X$ is countable by Theorem 3.2 of \cite{Gabr-LCS-Ascoli} which states that $L(X)$ is not an Ascoli space for every uncountable discrete space $X$. \qed
\end{proof}

Below we prove Theorem \ref{t:Ascoli-L(X)}.

\medskip
{\em Proof of Theorem \ref{t:Ascoli-L(X)}.} 
If $X$ is countable and discrete, then $L(X)$ is sequential (see \cite{Gabr}) and hence an Ascoli space.

Conversely, assume that $L(X)$ is an Ascoli space. Lemma 2.7 of \cite{GGKZ} states that if $H$ is a dense subgroup of a topological group $G$ and has the Ascoli property, then also $G$ is an Ascoli space. Since $L(X)$ is a dense subspace of $L(\mu X)$ by Theorem \ref{t:Free-complete-L}, the last fact implies that $L(\mu X)$ is also an Ascoli space. Now Proposition \ref{p:Ascoli-L(X)} implies that $\mu X$ is a countable discrete space. Thus also $X$ is countable and discrete. \qed



\bibliographystyle{amsplain}

\begin{thebibliography}{10}


\bibitem{Banakh-Survey}
T. Banakh, Fans and their applications in General Topology, Functional Analysis and Topological Algebra, available in arXiv:1602.04857.

\bibitem{BG}
T. Banakh, S. Gabriyelyan,  On the $\CC$-stable closure of the class of (separable) metrizable spaces, Monatshefte Math.   \textbf{180} (2016), 39--64.

\bibitem{Eng}
R.~Engelking, \emph{ General Topology}, Heldermann Verlag, Berlin, 1989.


\bibitem{Gabr}
S. Gabriyelyan, The $k$-space property for  free locally convex spaces, Canadian Math. Bull. \textbf{57} (2014), 803--809.




\bibitem{Gabr-LCS-Ascoli}
S.~Gabriyelyan, On the Ascoli property for locally convex spaces, Topology Appl. \textbf{230} (2017), 517--530.


\bibitem{Gabr-L(X)-Mackey}
S. Gabriyelyan, The Mackey problem for free locally convex spaces, Forum Math. accepted.


\bibitem{Gab-Respected}
S. Gabriyelyan, Maximally almost periodic groups and respecting properties, available in arXiv:1712.05521.


\bibitem{GGKZ}
S.~Gabriyelyan, J.  Greb\'{\i}k, J. K\c{a}kol, L. Zdomskyy,  The Ascoli property for function spaces, Topology Appl. \textbf{214} (2016), 35--50.



\bibitem{GKP}
S.~Gabriyelyan, J.~K{\c{a}}kol,  G. Plebanek, The Ascoli property for function spaces and the weak topology of Banach and Fr\'echet spaces, Studia Math. \textbf{233} (2016), 119--139.

\bibitem{GM}
S. Gabriyelyan, S.A. Morris, Free topological vector spaces, Topology Appl. \textbf{223} (2017), 30--49.

\bibitem{HerUsp}
S.~Hern\'{a}ndez, V.V.~Uspenski\u{\i}, Pontryagin duality for spaces of continuous functions, J. Math. Anal. Appl.  \textbf{242} (2000), 135--144.


\bibitem{Jar}
H.~Jarchow, \emph{Locally Convex Spaces}, B.G. Teubner, Stuttgart, 1981.


\bibitem{Kap}
S.~Kaplan, Extensions of the Pontryagin duality, I: Infinite products, Duke Math. J. \textbf{15} (1948), 649--658.

\bibitem{Mar}
A.A.~Markov,  On free topological groups, Dokl. Akad. Nauk SSSR \textbf{31} (1941), 299--301.

\bibitem{mcoy}
R.~McCoy, I.~Ntantu, \emph{Topological Properties of Spaces of Continuous Functions}, Lecture Notes in Math.  {1315}, 1988.

\bibitem{NaB}
L. Narici, E. Beckenstein,  \emph{Topological vector spaces}, Second Edition, CRC Press, New York, 2011.

\bibitem{Nob2}
N.~Noble, $k$-groups and duality, Trans. Amer. Math. Soc. \textbf{151} (1970), 551--561.

\

\bibitem{Rai}
D.A.~Ra\u{\i}kov, Free locally convex spaces for uniform spaces, Math. Sb. \textbf{63} (1964), 582--590.


\bibitem{Smith}
M.F.~Smith, The Pontrjagin duality theorem in linear spaces, Ann. Math. \textbf{56} (1952), 248--253.

\bibitem{Usp2}
V.V.~Uspenski\u{\i}, Free topological groups of metrizable spaces,  Math. USSR-Izv. \textbf{37} (1991), 657--680.


\end{thebibliography}

\end{document}